\documentclass[graybox]{svmult}
\usepackage{latexsym,amssymb,amsmath}
\usepackage[justification=centering]{caption}
\usepackage{tikz-cd}

\def\demo{\noindent{\it Proof. }}
\DeclareMathOperator{\depth}{depth}

\usepackage{mathptmx}       
\usepackage{helvet}         
\usepackage{courier}        
\usepackage{type1cm}        

\usepackage{makeidx}         
\usepackage{graphicx}        
\usepackage{multicol}        
\usepackage[bottom]{footmisc}

\makeindex             

\begin{document}

\title*{Symbolic powers of monomial ideals and Cohen-Macaulay vertex-weighted digraphs}
\titlerunning{Monomial ideals and Cohen-Macaulay digraphs} 

\author{Philippe Gimenez, Jos\'e Mart\'\i nez-Bernal, Aron Simis,
Rafael H. Villarreal, and Carlos E. Vivares\\
\ \ \\ 
\ \ \it Dedicated to Professor Antonio Campillo on the occasion of his
$65$th birthday}
\authorrunning{Gimenez, Mart\'\i nez-Bernal, Simis, Villarreal,
Vivares} 

\institute{Philippe Gimenez \at Instituto de Investigaci\'on en Matem\'aticas de la Universidad de
Valladolid (IMUVA), 
Facultad de Ciencias, 47011 Valladolid, Spain, 
\email{pgimenez@agt.uva.es} \and Jos\'e Mart\'\i nez-Bernal \at 
Departamento de
Matem\'aticas, Centro de Investigaci\'on y de Estudios Avanzados del
IPN, Apartado Postal 14--740, 07000 Mexico City, \email{jmb@math.cinvestav.mx} \and Aron
Simis \at Departamento de Matem\'atica, Universidade Federal de
Pernambuco, 50740-560 Recife, PE, Brazil, \email{aron@dmat.ufpe.br} \and 
Rafael H. Villarreal \at Departamento de
Matem\'aticas, Centro de Investigaci\'on y de Estudios Avanzados del
IPN, Apartado Postal 14--740, 07000 Mexico City,  
\email{vila@math.cinvestav.mx} \and 
Carlos E. Vivares \at Departamento de
Matem\'aticas, Centro de Investigaci\'on y de Estudios Avanzados del
IPN, Apartado Postal 14--740, 07000 Mexico City, 
\email{cevivares@math.cinvestav.mx} 
}

\maketitle

\abstract{In this paper we study irreducible representations
and symbolic Rees algebras of monomial ideals. Then we examine edge
ideals associated to vertex-weighted oriented graphs. 
These are digraphs 
having no oriented cycles of length two with
weights on the vertices. 
For a monomial ideal with no embedded
primes we classify the normality of its symbolic Rees algebra
in terms of its primary components. 
If the primary components of a
monomial ideal are normal, we present a simple procedure to compute
its symbolic Rees algebra using Hilbert bases, and give necessary and
sufficient conditions for the equality between its ordinary and
symbolic powers. We give an effective characterization 
of the Cohen--Macaulay vertex-weighted oriented forests. For edge
ideals of transitive 
weighted oriented graphs we show that Alexander duality holds. It is 
shown that edge ideals of weighted acyclic tournaments are Cohen--Macaulay
and satisfy Alexander duality.}

\section{Introduction}\label{intro-section}

Let $R=K[x_1,\ldots,x_n]$ be a polynomial ring over a field $K$ and
let $I\subset R$ be a monomial ideal. 
The {\it Rees algebra\/} of $I$ is
$$
R[It]:=R\oplus It\oplus\cdots\oplus I^{k}t^k\oplus\cdots
\subset R[t],
$$
where $t$ is a new variable, and the 
{\it symbolic Rees algebra} of $I$ is
$$
R_s(I):=R\oplus
I^{(1)}t\oplus\cdots\oplus I^{(k)}t^k\oplus\cdots\subset R[t],
$$
where $I^{(k)}$ is the $k$-th symbolic power of $I$ (see
Definition~\ref{symbolic-power-def}). 

One of the early works on symbolic powers of monomial ideal is
\cite{aron-hoyos}.  Symbolic powers of ideals and edge ideals of
graphs where studied in \cite{bahiano}. A method to compute symbolic
powers of radical ideals in characteristic zero is given in
\cite{aron-symbolic}.  

In Section~\ref{irreducible-deco} we 
recall the notion of irreducible decomposition of a monomial ideal
and prove 
that the exponents of the variables that occur in the minimal
generating set of a monomial ideal $I$ are exactly the exponents of the
variables that occur in the minimal generators of the irreducible
components of $I$ (Lemma~\ref{duality-of-exponents}). This 
result indicates that the well known Alexander duality for squarefree
monomial ideals could also hold for other families of monomial ideals.

We give algorithms to compute the symbolic powers of monomial 
ideals using {\it Macaulay\/}$2$ \cite{mac2}  
(Lemma~\ref{anoth-one-char-spow-general},
Remarks~\ref{symbolic-powers-algorithm} and
\ref{Tai-symbolic-powers-algorithm}). For a monomial ideal with no
embedded 
primes we classify the normality of its symbolic Rees algebra
in terms of the normality of its primary components
(Proposition~\ref{aug5-17}). 

The normality of a monomial ideal is well understood from the
computational point of view. If $I$ is minimally
generated by $x^{v_1},\ldots,x^{v_r}$ and $A$ is the matrix with
column vectors $v_1^{\rm t},\ldots,v_r^{\rm t}$, then $I$ is normal if and only if
the system $xA\geq\mathbf{1};x\geq 0$ has the integer rounding
property \cite[Corollary 2.5]{poset}. 
The normality of $I$ can be determined using the 
program {\it Normaliz\/} \cite{Normaliz}. For the normality of
monomial ideals of dimension $2$ see \cite{crispin-quinonez,icmi} and
the references therein.

To compute the generators of the symbolic Rees algebra of a monomial
ideal one can use the
algorithm in the proof of 
\cite[Theorem~1.1]{herzog-hibi-trung}. 
If the primary components of a
monomial ideal are normal, we present a procedure that 
computes the generators of its symbolic Rees algebra using Hilbert bases
and {\it Normaliz\/} \cite{Normaliz} (Proposition~\ref{aug6-17}, 
Example~\ref{symbolic-normaliz}), and give necessary and
sufficient conditions for the equality between its ordinary and
symbolic powers (Corollary~\ref{oct16-17}). 

In Section~\ref{cm-digraphs} we study edge ideals of weighted
oriented graphs. A {\it directed graph\/} or {\it digraph\/} $\mathcal{D}$ consists of a finite set 
$V({\mathcal D})$ of vertices, together with a prescribed collection
$E({\mathcal D})$ of ordered pairs of distinct points called {\it
edges\/} or {\it arrows\/}. An
{\it oriented graph\/} is a digraph having
no oriented cycles of length two. In other words an oriented graph
$\mathcal{D}$ is a simple graph $G$ together with an orientation of its
edges. We call $G$ the {\it underlying graph\/} of $\mathcal{D}$. If a
digraph $\mathcal{D}$ is endowed with a function $d\colon 
V(\mathcal{D})\rightarrow\mathbb{N}_+$, where
$\mathbb{N}_+:=\{1,2,\ldots\}$, we call $\mathcal{D}$ a
{\it vertex-weighted digraph\/}. 

Edge ideals of edge-weighted graphs were introduced and studied by Paulsen and
Sather-Wagstaff \cite{pausen-sean-sather-wagstaff}. In this work 
we consider edge ideals of graphs which are oriented and have weights
on the vertices. In what follows by a weighted
oriented graph we shall always mean a vertex-weighted oriented graph. 

Let $\mathcal{D}$ be a
vertex-weighted digraph with vertex set $V(\mathcal{D})=\{x_1,\ldots,x_n\}$.
The weight $d(x_i)$ of $x_i$ is denoted simply by $d_i$. 
The {\it edge ideal\/} of $\mathcal{D}$, denoted 
$I(\mathcal{D})$, is the ideal of $R$ given by
$$
I(\mathcal{D}):=(x_ix_j^{d_j}\, \vert\, (x_i,x_j)\in E(\mathcal{D})).
$$

If a vertex $x_i$ of $\mathcal{D}$ is a {\it source\/} (i.e., has
only arrows leaving $x_i$) we shall always 
assume $d_i=1$ because in this case the definition of $I(\mathcal{D})$
 does not depend on the weight of $x_i$. In the special case when
$d_i=1$ for all $i$, we recover the edge ideal of the graph $G$ 
which has been extensively studied in the literature
\cite{Dao-Huneke-schweig,francisco-ha-mermin,graphs-rings,HaM,Herzog-Hibi-book,
edge-ideals,ITG,chapter-vantuyl,Vi2,monalg-rev}. 
A vertex-weighted digraph $\mathcal{D}$ is called
{\it Cohen--Macaulay\/} (over the field $K$) if $R/I(\mathcal{D})$ is a
Cohen--Macaulay 
ring. 

Using a result of \cite{depth-monomial}, 
we answer a question of Aron Simis and a related question of Antonio
Campillo by showing that an oriented graph $\mathcal{D}$ is Cohen--Macaulay if and
only if the oriented graph $\mathcal{U}$, obtained from $\mathcal{D}$ by replacing 
each weight $d_i>3$ with $d_i=2$, is Cohen--Macaulay
(Corollary~\ref{oct20-17}). Seemingly, this ought to somewhat
facilitate the verification of this property. 

It turns out that edge ideals of weighted acyclic tournaments are Cohen--Macaulay
and satisfy Alexander duality (Corollaries~\ref{apr22-17} and \ref{jun1-17}). For transitive 
weighted oriented graphs it is shown that Alexander duality holds 
(Theorem~\ref{alexander-duality-transitive}). Edge
ideals of weighted digraphs arose in the theory of Reed-Muller codes as initial
ideals of vanishing ideals of projective spaces over finite
fields \cite{carvalho-lopez-lopez,Ha-Lin-Morey-Reyes-Vila,hilbert-min-dis}. 

A major result of Pitones, Reyes and
Toledo \cite{PRT} shows an explicit combinatorial expression for the irredundant 
decomposition of $I(\mathcal{D})$ as a finite intersection of irreducible
monomial ideals (Theorem~\ref{pitones-reyes-toledo}). We will
use their result to prove the following explicit combinatorial classification of
all Cohen--Macaulay weighted oriented forests.

\noindent {\bf Theorem~\ref{cm-weighted-oriented-trees}}{\it\  
Let $\mathcal{D}$ be a weighted oriented forest without isolated
vertices and let $G$ be its underlying forest. The following
conditions are equivalent:
\begin{enumerate}
\item[$(a)$] $\mathcal{D}$ is Cohen--Macaulay.
\item[$(b)$] $I(\mathcal{D})$ is unmixed, that is, all its associated
primes have the same height.
\item[$(c)$] $G$ has a 
perfect matching $\{x_1,y_1\},\ldots,\{x_r,y_r\}$ so that 
$\deg_G(y_i)=1$ for $i=1,\ldots, r$ and $d(x_i)=d_i=1$ if
$(x_i,y_i)\in E(\mathcal{D})$.
\end{enumerate}}
 
All rings considered here are Noetherian. For all unexplained
terminology and additional information,  we refer to 
\cite{digraphs} for the theory of digraphs, 
and \cite{graphs-rings,Herzog-Hibi-book,edge-ideals,monalg-rev} for the theory of
edge ideals of graphs and monomial ideals.

\section{Irreducible decompositions and symbolic powers}\label{irreducible-deco}

In this section we study irreducible representations of monomial
ideals and various aspects of symbolic Rees algebras of monomial ideals. Here we continue to employ
the notation and definitions used in Section~\ref{intro-section}.

Recall that an ideal $L$ of a Noetherian ring $R$ is called {\it irreducible} if
$L$ cannot be written as an intersection of two ideals of $R$ that
properly contain $L$. Let $R=K[x_1,\ldots,x_n]$ be a polynomial ring
over a field $K$. Up to permutation of variables the irreducible 
monomial ideals of $R$ are of the form
$$
(x_1^{a_1},\ldots,x_r^{a_r}),
$$
where $a_1,\ldots,a_r$ are positive integers. According to
\cite[Theorem~6.1.17]{monalg-rev} any monomial ideal $I$
of $R$ has a {\it unique\/} irreducible decomposition:
$$ 
I=I_1\cap\cdots\cap I_m,
$$
where $I_1,\ldots,I_m$ are irreducible monomial ideals and 
$I\neq\cap_{i\neq j}I_i$ for  $j=1,\ldots,m$, that is, 
this decomposition is irredundant. The ideals $I_1,\ldots,I_m$ are called
 the {\it irreducible components\/} of $I$. 

By \cite[Proposition~6.1.7]{monalg-rev} a monomial ideal
$\mathfrak{I}$ is 
a primary ideal if and only if, after permutation of the 
variables, it has the form:
\begin{equation}\label{primary-monomial-ideal}
\mathfrak{I}=(x_1^{a_1},\ldots,x_r^{a_r},x^{b_1},\ldots,x^{b_s}),
\end{equation}
where $a_i\geq 1$ and $\cup_{i=1}^s{\rm supp}(x^{b_i})\subset
\{x_1,\ldots,x_r\}$. Thus if $\mathfrak{I}$ is a
monomial primary ideal, then $\mathfrak{I}^k$ 
is a primary ideal for $k\geq 1$. 
Since irreducible ideals are
primary, the irreducible decomposition of $I$ is a 
primary decomposition of $I$.  
Notice that the irreducible decomposition of $I$ is not necessarily
a minimal primary decomposition, that is, $I_i$ and $I_j$ could have the same radical for
$i\neq j$. If $I$ is a squarefree monomial ideal, its 
irreducible decomposition is minimal. For edge ideals of
weighted oriented graphs one also has that their irreducible
decompositions are minimal \cite{PRT}.

\begin{definition} An irreducible monomial ideal $L\subset R$ is 
called a {\it minimal irreducible\/} ideal of $I$ if
$I\subset L$ and for any
irreducible monomial ideal $L'$ such that $I\subset L'\subset L$ one
has that $L=L'$. 
\end{definition}


\begin{proposition}\label{apr23-17} If $I=I_1\cap\cdots\cap I_m$ is the irreducible
decomposition of a monomial ideal $I$, then $I_1,\ldots,I_m$ are 
the minimal irreducible monomial ideals of $I$. 
\end{proposition}

\begin{proof} Let $L$ be an irreducible ideal that contains $I$. 
Then $I_i\subset L$ for some $i$. Indeed if $I_i\not\subset L$ for all
$i$, for each $i$ pick $x_{j_i}^{a_{j_i}}$ in $I_i\setminus L$.
Since $I\subset L$, setting $x^a={\rm lcm}\{x_{j_i}^{a_{j_i}}\}_{i=1}^m$ and writing
$L=(x_{k_1}^{c_{k_1}},\ldots,x_{k_\ell}^{c_{k_\ell}})$, it follows
that $x^a$ is in $I$ and $x_{j_i}^{a_{j_i}}$ is a multiple of $x_{k_t}^{c_{k_t}}$ for some
$1\leq i\leq m$ and $1\leq t\leq \ell$. Thus $x_{j_i}^{a_{j_i}}$ is in
$L$, a contradiction. Therefore if $L$ is minimal one has $L=I_i$ for
some $i$. To complete the proof notice that $I_i$ is a minimal irreducible
monomial ideal of $I$ for all $i$. This follows from the first part of
the proof using that $I=I_1\cap\cdots\cap I_m$ is an irredundant
decomposition. 
\qed
\end{proof}

The unique minimal set of generators of a monomial ideal $I$, 
consisting of monomials, is denoted by $G(I)$. The next result tells us
that in certain cases we may have a sort of Alexander duality
obtained by switching the roles of minimal generators and irreducible
components \cite[Theorem 6.3.39]{monalg-rev} (see
Example~\ref{example-alex-duality} and
Theorem~\ref{alexander-duality-transitive}).    

\begin{lemma}\label{duality-of-exponents} Let $I$ be a monomial ideal of $R$, with 
$G(I)=\{x^{v_1},\ldots,x^{v_r}\}$ and
$v_i=(v_{i1},\ldots,v_{in})$ for $i=1,\ldots,r$, and let
$I=I_1\cap\cdots\cap I_m$ be its irreducible decomposition. Then
$$
V:=\{x_j^{v_{ij}}\vert\, v_{ij}\geq 1\}=G(I_1)\cup\cdots\cup G(I_m).
$$ 
\end{lemma}

\begin{proof} ``$\subset$'': Take $x_j^{v_{ij}}$ in $V$, without loss
of generality we may assume $i=j=1$. We proceed by contradiction
assuming that $x_1^{v_{11}}$ is not in $\cup_{i=1}^mG(I_i)$. Setting 
$M=x_1^{v_{11}-1}x_2^{v_{12}}\cdots x_n^{v_{1n}}$, notice that $M$ is
in $I$. Indeed for any $I_j$ not containing $x_2^{v_{12}}\cdots
x_n^{v_{1n}}$, one has that $x_1^{v_{11}}$ is in $I_j$ because
$x^{v_1}$ is in $I$. Thus there is $x_1^{c_j}$ in $G(I_j)$ such that 
$v_{11}> c_j\geq 1$ because $x_1^{v_{11}}$ is not in $G(I_j)$. Thus
$M$ is in $I_j$. This proves that $M$ is in $I$, 
a contradiction to the minimality 
of $G(I)$ because this monomial that strictly divides one of the
elements of $G(I)$ cannot be in $I$. Thus $x_1^{v_{11}}$ is in $\cup_{i=1}^mG(I_i)$, as
required.

``$\supset$'': Take $x_j^{a_j}$ in $G(I_i)$ for some $i,j$, without
loss of generality we may assume that $i=j=1$ and
$G(I_1)=\{x_1^{a_1},\ldots,x_\ell^{a_\ell}\}$. We proceed by
contradiction assuming that $x_1^{a_1}\notin V$. Setting
$L=(x_1^{a_1+1},x_2^{a_2},\ldots,x_\ell^{a_\ell})$, notice that
$I\subset L$. Indeed take any monomial $x^{v_k}$ in $G(I)$ which is
not in $(x_2^{a_2},\ldots,x_\ell^{a_\ell})$. Then $x^{v_k}$ is a
multiple of $x_1^{a_1}$ because $I\subset I_1$. Hence $v_{k1}>a_1$
because $x_1^{a_1}\notin V$. Thus $x^{v_k}$ is in
$L$. This proves that $I\subset L\subsetneq I_1$, a contradiction to
the fact that $I_1$ is a minimal irreducible monomial ideal of $I$
(see Proposition~\ref{apr23-17}). \qed
\end{proof}

Let $I\subset R$ be a monomial ideal. The {\it Alexander
dual\/} of $I$, denoted $I^\vee$, is the ideal of $R$ generated by all 
monomials $x^a$, with $a=(a_1,\ldots,a_n)$, such that 
$\{x_i^{a_i}\vert\, a_i\geq 1\}$ is equal to $G(L)$ for some minimal irreducible
ideal $L$ of $I$. 
The {\it dual\/} of $I$, denoted $I^*$, is the
intersection of all ideals $(\{x_i^{a_i}\vert\, a_i\geq 1\})$ such
that $x^a\in G(I)$. Thus one has
$$
I^\vee=\left(\prod_{f\in G(I_1)}\!\!f,\ \ldots,\prod_{f\in
G(I_m)}\!\!f\right)\ \mbox{ and }\ I^*=\bigcap_{x^a\in
G(I)}(\{x_i^{a_i}\vert\, a_i\geq 1\}),
$$
where $I_1,\ldots,I_m$ are the irreducible components of $I$. 
If $I^*=I^\vee$, we say that
{\it Alexander duality\/} holds for $I$. There are other related ways
introduced by Ezra Miller
\cite{hosten-smith,ezra-miller-1,ezra-miller,cca} 
to define the Alexander dual of a monomial ideal . It is well known that 
$I^*=I^\vee$ for squarefree monomial ideals \cite[Theorem
6.3.39]{monalg-rev}.   

\begin{definition}\label{symbolic-power-def}\rm 
Let $I$ be an ideal of a ring $R$ and 
let $\mathfrak{p}_1,\ldots,{\mathfrak p}_r$ be  
the minimal primes of $I$. Given an integer $k\geq 1$, we define 
the $k$-th {\it symbolic power} of 
$I$ to be the ideal 
$$
I^{(k)}:=\bigcap_{i=1}^r\mathfrak{q}_i=\bigcap_{i=1}^r
(I^kR_{\mathfrak{p}_i}\cap
R),
$$
where $\mathfrak{q}_i$ is the ${\mathfrak p}_i$-primary component of
$I^k$. 
\end{definition}

In other words, one has $I^{(k)}=S^{-1}I^k\cap
R$, where $S=R\setminus\cup_{i=1}^r{\mathfrak p}_i$. An alternative
notion of symbolic 
power can 
be introduced 
using the whole set of associated
primes of $I$ instead 
(see, e.g., \cite{cooper-etal,symbolic-powers-survey}): 
$$
I^{\langle k\rangle} =\bigcap_{\mathfrak{p}\in {\rm Ass}(R/I)}(I^kR_\mathfrak{p}\cap
R)=\bigcap_{\mathfrak{p}\in {\rm maxAss}(R/I)}
(I^kR_\mathfrak{p}\cap
R),
$$
where ${\rm maxAss}(R/I)$ is the set of associated primes which are
maximal with respect to inclusion \cite[Lemmas~3.1 and 3.2]{cooper-etal}.
Clearly $I^k\subset I^{\langle k\rangle}\subset I^{(k)}$. 
If $I$ has no embedded primes, e.g. for radical ideals such as 
squarefree monomial ideals, the two last definitions of symbolic powers
coincide. An interesting problem is to give necessary and sufficient conditions
for the equality ``$I^k=I^{(k)}$ for $k \geq 1$''.

For prime ideals the $k$-th symbolic powers and
the $k$-th usual powers are not always equal. Thus the next lemma 
does not hold in general but the proof below shows that it will hold 
for an ideal $I$ in Noetherian 
ring $R$ under the assumption that
$\mathfrak{I}_i^{k}=\mathfrak{I}_i^{(k)}$ for $i=1,\ldots,r$. 
The next lemma is well known for radical monomial ideals
\cite[Propositions~3.3.24 and 7.3.14]{monalg}.

\begin{lemma}\label{anoth-one-char-spow-general}
Let $I\subset R$ be a monomial ideal and let 
$I=\mathfrak{I}_1\cap\cdots\cap\mathfrak{I}_r\cap\cdots\cap\mathfrak{I}_m$
be an irredundant minimal primary decomposition of $I$, where
$\mathfrak{I}_1,\ldots,\mathfrak{I}_r$ 
are the primary components associated to the minimal primes of $I$.
Then 
$$
I^{(k)}={\mathfrak I}_1^{k}\cap\cdots\cap {\mathfrak I}_r^{k}\ 
\mbox{ for }\ k\geq 1.
$$
\end{lemma}

\demo 
Let $\mathfrak{p}_1,\ldots,\mathfrak{p}_r$ be the
minimal primes of $I$. By
\cite[Proposition~6.1.7]{monalg-rev} any power of $\mathfrak{I}_i$ is
again a $\mathfrak{p}_i$-primary ideal (see
Eq.~(\ref{primary-monomial-ideal}) at the beginning of this section). Thus
$\mathfrak{I}_i^{k}=\mathfrak{I}_i^{(k)}$ for any $i,k$. Fixing 
integers $k\geq 1$ and $1\leq i\leq r$, let 
$$I^k=\mathfrak{q}_1\cap\cdots\cap \mathfrak{q}_r
\cap \cdots\cap \mathfrak{q}_s$$
be a primary decomposition of $I^k$, where 
$\mathfrak{q}_j$ is ${\mathfrak p}_j$-primary for $j\leq r$. 
Localizing at 
${\mathfrak p}_i$ yields 
$I^kR_{{\mathfrak p}_i}=\mathfrak{q}_iR_{{\mathfrak p}_i}$ and 
from
$I=\mathfrak{I}_1\cap\cdots\cap\mathfrak{I}_r\cap\cdots\cap\mathfrak{I}_m$  
one obtains: 
\[
I^kR_{{\mathfrak p}_i}=(IR_{{\mathfrak p}_i})^k=
(\mathfrak{I}_iR_{{\mathfrak p}_i})^k=
\mathfrak{I}_i^kR_{{\mathfrak p}_i}.
\]
Thus $\mathfrak{I}_i^kR_{{\mathfrak p}_i}=
\mathfrak{q}_iR_{{\mathfrak p}_i}$ and contracting to $R$ 
one has $\mathfrak{I}_i^{(k)}=\mathfrak{q}_i$. Therefore 
$$
I^{(k)}={\mathfrak I}_1^{(k)}\cap\cdots\cap {\mathfrak I}_r^{(k)}
={\mathfrak I}_1^{k}\cap\cdots\cap {\mathfrak I}_r^{k}.\eqno\qed
$$

It was pointed out to us 
by Ng\^{o} Vi\^{e}t Trung that
Lemma~\ref{anoth-one-char-spow-general} 
is a consequence of
\cite[Lemma~3.1]{herzog-hibi-trung}. This lemma also follows from 
\cite[Proposition~3.6]{cooper-etal}.

\begin{remark}\label{symbolic-powers-algorithm} 
To compute the $k$-th symbolic power $I^{(k)}$ of a monomial
ideal $I$ one can use the following procedure for {\em Macaulay\/}$2$ \cite{mac2}.
\begin{small}
\begin{verbatim}
SPG=(I,k)->intersect(for n from 0 to #minimalPrimes(I)-1 
list localize(I^k,(minimalPrimes(I))#n))
\end{verbatim}
\end{small}
\end{remark}

\begin{example}\label{oct23-17} Let $I$ be the ideal
$(x_2x_3,x_4x_5,x_3x_4,x_2x_5,x_1^2x_3,x_1x_2^2)$. Using the procedure
of Remark~\ref{symbolic-powers-algorithm} we obtain 
$I^{(2)}=I^2+(x_1x_2^2x_5,x_1x_2^2x_3)$.
\end{example}

\begin{remark}\label{symbolic-powers-algorithm-ass} 
If one uses ${\rm Ass}(R/I)$ to define the symbolic powers of a
monomial ideal $I$, the following function for {\em Macaulay\/}$2$
\cite{mac2} can be used to compute $I^{\langle k\rangle}$. 
\begin{small}
\begin{verbatim}
SPA=(I,k)->intersect(for n from 0 to #associatedPrimes(I)-1 
list localize(I^k,(associatedPrimes(I))#n))
\end{verbatim}
\end{small}
\end{remark}

\begin{example}\label{example-symbolic-ass} Let $I$ be the ideal 
$(x_1x_2^2,x_3x_1^2,x_2x_3^2)$. Using the procedures
of Remarks~\ref{symbolic-powers-algorithm} and
\ref{symbolic-powers-algorithm-ass}, we obtain
$$
I^{(1)}=I+(x_1x_2x_3)\ \mbox{ and }\ I^{\langle 1\rangle}=I.
$$
\end{example}

\begin{remark} The following formula is useful to study 
the symbolic powers $I^{\langle k\rangle}$ of a monomial ideal $I$ 
\cite[Proposition~3.6]{cooper-etal}:
$$
I^kR_\mathfrak{p}\cap R=(IR_\mathfrak{p}\cap R)^k\ \mbox{ for }\
\mathfrak{p}\in {\rm Ass}(R/I) \ \mbox{ and }\ k\geq 1.
$$
\end{remark}

\begin{definition}\rm An ideal $I$ of a ring $R$ is 
called {\em normally torsion-free}
 if ${\rm Ass}(R/I^k)$ is contained in ${\rm Ass}(R/I)$ for all 
$k\geq 1$. 
\end{definition}

\begin{remark}\label{oct10-17} Let $I$ be an ideal of a ring $R$. If $I$ has no embedded 
primes, then $I$ is normally torsion-free if and only 
if $I^k=I^{(k)}$ for all $k\geq 1$.
\end{remark}

\begin{lemma}{\rm\cite[Lemma 5, Appendix 6]{zariski-samuel-ii}}
\label{ntr-complete-intersection}
 Let $I\subset R$ be an ideal generated by a regular sequence.
Then $I^k$ is unmixed for $k\geq 1$. In particular $I^{k}=I^{(k)}$ 
for $k\geq 1$. 
\end{lemma}

One can also compute the symbolic powers of vanishing ideals of finite
sets of reduced projective points using
Lemma~\ref{anoth-one-char-spow-general} 
because these ideals  are 
intersections of finitely many
prime ideals that are complete intersections. It is well known that
complete intersections are normally torsion-free
(Lemma~\ref{ntr-complete-intersection}).

\begin{remark}{(Jonathan O'Rourke)}\label{Tai-symbolic-powers-algorithm} 
 If $I$ is a radical 
ideal of $R$ and all associated primes of $I$ are normally
torsion-free, 
then the $k$-th symbolic power of $I$ can be computed using the
following 
procedure for {\em Macaulay\/}$2$ \cite{mac2}.
\begin{small}
\begin{verbatim}
SP1 = (I,k) ->  (temp = primaryDecomposition I; 
temp2 = ((temp_0)^k); for i from 1 to #temp-1 do(temp2 =                               
     intersect(temp2,(temp_i)^k)); return temp2)                  
\end{verbatim}
\end{small}
\end{remark}

\begin{example}\label{aron-valladolid} Let $\mathbb{X}$ be the set
$\{[e_1],[e_2],[e_3],[e_4],[(1,1,1,1)]\}$ of $5$ points in general
linear position in $\mathbb{P}^3$, over the field $\mathbb{Q}$, 
where $e_i$ is the $i$-th unit vector, and let $I=I(\mathbb{X})$ be
its vanishing ideal. Using {\em Macaulay\/}$2$ \cite{mac2}
and Remark~\ref{Tai-symbolic-powers-algorithm} we obtain
$$
I=(x_2x_4-x_3x_4, x_1x_4-x_3x_4, x_2x_3-x_3x_4, x_1x_3-x_3x_4,
x_1x_2-x_3x_4),
$$ 
$I^2=I^{(2)}$, $I^3\neq I^{(3)}$ and $I$ is a Gorenstein ideal. This
example (in greater generality)
has been used in \cite[proof of Proposition~4.1 and
Remark~4.2(2)]{Aron-Valladolid}.
\end{example}

\begin{proposition}{\rm\cite{herzog-hibi-trung}}\label{symbolic-fg-monomial}
If $I\subset R$ is a monomial ideal, then  
the symbolic Rees algebra $R_s(I)$ of $I$ is a finitely generated $K$-algebra. 
\end{proposition}

\begin{proof} 
It follows at once from Lemma~\ref{anoth-one-char-spow-general} 
and \cite[Corollary~1.3]{herzog-hibi-trung}.\qed
\end{proof}

To compute the generators of the symbolic Rees algebra of a monomial
ideal one can use the
procedure given in the proof of
\cite[Theorem~1.1]{herzog-hibi-trung}. Another method will be 
presented in this section that works when the primary components 
are normal.

\begin{remark} The symbolic Rees 
algebra of a monomial ideal $I$ is finitely generated if one uses
the associated primes of $I$ to define symbolic powers. This follows
from \cite[Corollary~1.3]{herzog-hibi-trung} and 
the following formula \cite[Theorem~3.7]{cooper-etal}:
$$
I^{\langle k\rangle}=\bigcap_{\mathfrak{p}\in {\rm maxAss}(R/I)}
(IR_\mathfrak{p}\cap
R)^k\ \mbox{ for }\ k\geq 1. 
$$
\end{remark}

\begin{corollary} If $I$ is a monomial ideal, then $R_s(I)$ is
Noetherian and there is an integer $k\geq 1$ such that
$[I^{(k)}]^i=I^{(ik)}$ for $i\geq 1$.
\end{corollary}

\begin{proof} It follows at once from \cite[p.~80, Lemma~2.1]{GoNi} or
by a direct argument using 
Proposition~\ref{symbolic-fg-monomial}. \qed
\end{proof}

For convenience of notation in what follows we will often assume that
monomial  ideals have no embedded primes but some of the results can
be stated and proved for general monomial ideals.   

\begin{proposition}\label{aug5-17} Let $I\subset R$ be a monomial ideal without embedded
primes and let 
$I=\cap_{i=1}^r\mathfrak{I}_i$ be its minimal
irredundant primary decomposition. Then $R_s(I)$ is normal if and 
only if $R[\mathfrak{I}_it]$ is normal for all $i$.
\end{proposition}

\begin{proof} $\Rightarrow$): Since $R_s(I)$ is Noetherian and normal it is a Krull
domain by a theorem of Mori and Nagata \cite[p.~296]{Mat}. 
Therefore, by \cite[Lemma~2.5]{simis-trung}, we get that
$R_{\mathfrak{p}_i}[I_{\mathfrak{p}_i}t]=
R_{\mathfrak{p}_i}[(\mathfrak{I}_i)_{\mathfrak{p}_i}t]$
is normal.  Let $\mathfrak{p}_i$ be the radical of $\mathfrak{I}_i$.
Any power of $\mathfrak{I}_i$ is  
a $\mathfrak{p}_i$-primary ideal. This follows from
\cite[Proposition~6.1.7]{monalg-rev} (see
Eq.~(\ref{primary-monomial-ideal}) at the beginning of this section).  
Hence it is seen that 
$R_{\mathfrak{p}_i}[(\mathfrak{I}_i)_{\mathfrak{p}_i}t]\cap
R[t]=R[\mathfrak{I}_it]$. 
As $R[t]$ is normal it follows 
that $R[\mathfrak{I}_it]$ is normal.

$\Leftarrow$): By Lemma~\ref{anoth-one-char-spow-general} one has
$\cap_{i=1}^rR[\mathfrak{I}_it]=R_s(I)$. As $R[\mathfrak{I}_it]$ and
$R_s(I)$ have the same field of quotients it follows that $R_s(I)$ is
normal. \qed
\end{proof}

In general, even for monomial ideals without embedded primes, normally
torsion-free ideals may not be normal. For instance $I=(x_1^2,x_2^2)$
is normally torsion-free and is not normal. 
As a consequence of Proposition~\ref{aug5-17} one recovers the
following well known result.

\begin{corollary} Let $I$ be a squarefree monomial ideal. Then
$R_s(I)$ is normal and $R[It]$ is normal if $I$ is normally 
torsion-free. 
\end{corollary} 

Let $I$ be a monomial ideal and let $G(I)=\{x^{v_1},\ldots,x^{v_m}\}$
be its minimal set of generators. We set
$$
\mathcal{A}_I=\{e_1,\ldots,e_n,(v_1,1),\ldots,(v_m,1)\},
$$
where $e_1,\ldots,e_n$ belong to $\mathbb{Z}^{n+1}$, and denote by
$\mathbb{R}_+(I)$ or $\mathbb{R}_+\mathcal{A}_I$ (resp. 
$\mathbb{N}\mathcal{A}_I$) the cone (resp. semigroup) generated by
$\mathcal{A}_I$. The integral closure of $R[It]$ is  
given by $\overline{R[It]}=K[\mathbb{R}_+(I)\cap\mathbb{Z}^{n+1}]$. Recall that
a finite set $\mathcal{H}$ is called a {\it Hilbert basis\/} if 
$\mathbb{N}\mathcal{H}=
\mathbb{R}_+\mathcal{H}\cap\mathbb{Z}^{n+1}$, and that $R[It]$ is
normal if and only if $\mathcal{A}_I$ is a Hilbert basis
\cite[Proposition~14.2.3]{monalg-rev}. 

Let $C\subset\mathbb{R}^{n+1}$ be a rational polyhedral cone. A finite 
set $\mathcal{H}$ is called a
Hilbert basis of $C$ if $C =\mathbb{R}_+\mathcal{H}$ and
$\mathcal{H}$ is a Hilbert basis. A Hilbert basis of
$C$ is minimal if it does not strictly contain any other Hilbert
basis of $C$. For pointed cones there is unique minimal Hilbert 
basis \cite[Theorem~1.3.9]{monalg-rev}. 

If the primary components of a
monomial ideal are normal, the next result gives a simple procedure to compute
its symbolic Rees algebra using Hilbert bases.

\begin{proposition}\label{aug6-17} Let $I$ be a monomial ideal without embedded
primes and let $I=\cap_{i=1}^r\mathfrak{I}_i$ be its minimal
irredundant primary decomposition. If $R[\mathfrak{I}_it]$ is normal
for all $i$ and $\mathcal{H}$ is the Hilbert basis of
the polyhedral cone $\cap_{i=1}^r\mathbb{R}_+(\mathfrak{I}_i)$, then 
$R_s(I)$ is $K[\mathbb{N}\mathcal{H}]$, the semigroup ring of
$\mathbb{N}\mathcal{H}$.
\end{proposition}

\demo As
$R[\mathfrak{I}_it]=K[\mathbb{N}\mathcal{A}_{\mathfrak{I}_i}]$ 
is normal for $i=1,\ldots,r$, the semigroup
$\mathbb{N}\mathcal{A}_{\mathfrak{I}_i}$ is equal to 
$\mathbb{R}_+(\mathfrak{I}_i)\cap\mathbb{Z}^{n+1}$ for
$i=1,\ldots,r$. Hence, by Lemma~\ref{anoth-one-char-spow-general}, we get 
\begin{eqnarray*}
R_s(I)&=&\cap_{i=1}^rR[\mathfrak{I}_it]=\cap_{i=1}^rK[\mathbb{N}\mathcal{A}_{\mathfrak{I}_i}]=
K[\cap_{i=1}^r\mathbb{N}\mathcal{A}_{\mathfrak{I}_i}]\\
&=&K[\mathbb{R}_+(\mathfrak{I}_1)\cap\cdots
\cap\mathbb{R}_+(\mathfrak{I}_r)\cap\mathbb{Z}^{n+1}]=K[\mathbb{N}\mathcal{H}].
\ \qed
\end{eqnarray*}

\begin{definition} The rational polyhedral cone
$\cap_{i=1}^r\mathbb{R}_+(\mathfrak{I}_i)$ is called the {\it Simis
cone} of $I$  and is denoted by ${\rm Cn}(I)$.
\end{definition}

For squarefree monomial ideals the Simis cone was introduced in
\cite{normali}. In particular from Proposition~\ref{aug6-17} we
recover \cite[Theorem~3.5]{normali}.

\begin{example}\label{symbolic-normaliz} The ideal $I=(x_2x_3,\, x_4x_5,\,
x_3x_4,\, x_2x_5,\, x_1^2x_3,\, x_1x_2^2)$ satisfies the hypothesis
of Proposition~\ref{aug6-17}. Using {\it Normaliz\/} \cite{Normaliz} we obtain that 
the minimal Hilbert basis of the Simis cone is:
\begin{verbatim}
18 Hilbert basis elements:
 0 0 0 0 1 0     1 2 0 0 0 1 
 0 0 0 1 0 0     2 0 1 0 0 1
 0 0 1 0 0 0     1 2 0 0 1 2
 0 1 0 0 0 0     1 2 1 0 0 2
 1 0 0 0 0 0     2 2 1 0 1 3
 0 0 0 1 1 1     2 2 2 0 0 3
 0 0 1 1 0 1     2 4 1 0 2 5
 0 1 0 0 1 1     2 4 2 0 1 5
 0 1 1 0 0 1     2 4 3 0 0 5
\end{verbatim}

Hence $R_s(I)$ is generated by the monomials corresponding to these 
vectors.
\end{example}

Let $I$ be an ideal of $R$. The equality 
``$I^k=I^{(k)}$ for $k \geq 1$'' holds if and only if $I$ has no 
embedded primes and is normally
torsion-free (see Remark~\ref{oct10-17}). We refer the reader to 
\cite{symbolic-powers-survey} for a recent survey on symbolic powers of ideals.   

In \cite[Corollary~3.14]{clutters} it is shown that a squarefree
monomial ideal $I$ is normally torsion-free if and only if the 
corresponding hypergraph satisfies the max-flow min-cut property. 
As an application we present a classification of the equality
between ordinary and 
symbolic powers for a family of monomial ideals. 

\begin{corollary}\label{oct16-17} Let $I$ be a monomial ideal without embedded
primes and let $\mathfrak{I}_1,\ldots,\mathfrak{I}_r$  be its primary
components. If $R[\mathfrak{I}_it]$ is normal
for all $i$, then $I^k=I^{(k)}$ for $k\geq 1$ if and 
only if\, ${\rm Cn}(I)=\mathbb{R}_+(I)$ and $R[It]$ is normal. 
\end{corollary}

\begin{proof} $\Rightarrow$): As $R_s(I)=R[It]$, by
Proposition~\ref{aug6-17}, $R[It]$ is normal. Therefore one has
$$
K[{\rm
Cn}(I)\cap\mathbb{Z}^{n+1}]=R_s(I)=R[It]=
\overline{R[It]}=K[\mathbb{R}_+(I)\cap\mathbb{Z}^{n+1}].
$$
\quad Thus ${\rm Cn}(I)=\mathbb{R}_+(I)$.

$\Leftarrow$): By the proof of Proposition~\ref{aug6-17} one has
$R_s(I)=K[{\rm Cn}(I)\cap\mathbb{Z}^{n+1}]$. Hence
$$
R_s(I)=K[{\rm Cn}(I)\cap\mathbb{Z}^{n+1}]=K[\mathbb{R}_+(I)\cap
\mathbb{Z}^{n+1}]=\overline{R[It]}.
$$
\quad As $R[It]$ is normal, we get $R_s(I)=R[It]$, that is,
$I^k=I^{(k)}$ for $k\geq 1$. 
\qed
\end{proof}

\section{Cohen--Macaulay weighted oriented trees}\label{cm-digraphs}

In this section we show that edge ideals of transitive 
weighted oriented graphs satisfy Alexander duality. It turns out
that edge ideals of weighted acyclic tournaments are Cohen--Macaulay
and satisfy Alexander duality. Then we classify all Cohen--Macaulay weighted oriented
forests. Here we continue to employ
the notation and definitions used in Sections~\ref{intro-section} 
and \ref{irreducible-deco}.

Let $G$ be a graph with vertex set $V(G)$. A subset $C\subset V(G)$  
is a {\it minimal vertex cover\/}\index{vertex!cover!minimal} of 
$G$ if:\index{minimal!vertex cover!of a graph} 
(i) every edge of $G$ is incident with at least one vertex in $C$, 
and (ii) there is no proper subset of $C$ with the first 
property. If $C$ satisfies condition (i) only, then $C$ is 
called a {\it vertex cover\/}\index{vertex!cover of a graph} of $G$.

Let $\mathcal{D}$ be a weighted oriented graph 
with underlying graph
$G$. Next we recall a combinatorial description of the irreducible
decomposition of $I(\mathcal{D})$. 

\begin{definition}{\cite{PRT}} Let $C$ be a vertex cover of $G$. 
Consider the set $L_1(C)$ of all $x\in C$ such that there is 
$(x,y)\in E(\mathcal{D})$ with $y\notin C$, the set $L_3(C)$ of all
$x\in C$ such that $N_G(x)\subset C$, and the set $L_2(C)=C\setminus(
L_1(C)\cup L_3(C))$, where $N_G(x)$ is the {\it neighbor} set of $x$ 
consisting of all $y\in V(G)$ such that $\{x,y\}$ is an edge of $G$. 
A vertex cover $C$ of $G$ is called a {\it strong
vertex cover} of $\mathcal{D}$ if $C$ is a minimal vertex cover of
$G$ or else for
all $x\in L_3(C)$ there is $(y,x)\in E(\mathcal{D})$ such that $y\in
L_2(C)\cup L_3(C)$ with $d(y)\geq 2$. 
\end{definition}

\begin{theorem}{\rm\cite{PRT}}\label{Pitones-Reyes} Let $\mathcal{D}$ be a weighted oriented graph. 
Then $L$ is a minimal irreducible monomial ideal of $I(\mathcal{D})$ if and 
only if there is a strong vertex cover of $\mathcal{D}$ such that
$$  
L=(L_1(C)\cup\{x_i^{d_i}\vert\, x_i\in L_2(C)\cup
L_3(C)\}).
$$
\end{theorem}

\begin{theorem}{\rm\cite{PRT}}\label{pitones-reyes-toledo} If
$\mathcal{D}$ is a weighted oriented  
graph and $\Upsilon(\mathcal{D})$ is the set of all strong vertex covers of 
$\mathcal{D}$, then the irreducible decomposition of $I(\mathcal{D})$
is 
$$
I(\mathcal{D})=\bigcap_{C\in\Upsilon(\mathcal{D})}I_C, 
$$
where $I_C=(L_1(C)\cup\{x_i^{d_i}\vert\, x_i\in L_2(C)\cup L_3(C)\})$. 
\end{theorem}

\begin{proof} This follows at once 
from Proposition~\ref{apr23-17} and Theorem~\ref{Pitones-Reyes}. \qed
\end{proof}

\begin{corollary}{\rm\cite{PRT}}\label{prt-main} Let $\mathcal{D}$ be a weighted 
oriented graph. Then $\mathfrak{p}$ is an associated prime of
$I(\mathcal{D})$ if and only if $\mathfrak{p}=(C)$ for some strong
vertex cover $C$ of $\mathcal{D}$.
\end{corollary}

\begin{example} Let $K$ be the field of rational numbers and 
let $\mathcal{D}$ be the weighted digraph of Fig.~\ref{digraph1} 
\begin{figure}
\centering
	\begin{tikzpicture}[line width=1.1pt,scale=0.95]
	\tikzstyle{every node}=[inner sep=0pt, minimum width=4.5pt]
	
	\draw [->] (-0.05,1.66)--(-1.49,.51); 
	\draw [->] (-1.05,-1.3)--(-1.5,.3); 
	\draw [->] (1.4,.4) --(-1.4,.4); 
	\draw [->] (-.9,-1.38)--(.9,-1.38); 
	\draw [->] (1.5,.4)--(1.01,-1.28); 
	\draw [->] (-0.95,-1.3)--(0,1.6); 
	
	\draw (-1,-1.38) node (v3) [draw, circle, fill=gray] {};
	\draw (1,-1.38) node (v4) [draw, circle, fill=gray] {};
	\draw (1.5,.4) node (v5) [draw, circle, fill=gray] {};
	\draw (0,1.7) node (v1)[draw, circle, fill=gray] {};
	\draw (-1.5,.4) node (v2) [draw, circle, fill=gray] {};

	\node at (-1.5,-1.38) {$x_{3}$};
	\node at (1.5,-1.38) {$x_{4}$};
	\node at (1.7,.8) {$x_{5}$};
	\node at (0,2.1) {$x_{1}$};
	\node at (-1.8,.8) {$x_{2}$};
	
	\node at (-.9,-1.7) {$d_3=1$};
	\node at (1.1,-1.7) {$d_4=2$};
	\node at (2.1,.10) {$d_5=1$};
	\node at (0.6,1.4) {$d_1=2$};
	\node at (-2.1,.10) {$d_2=2$};
	
	\end{tikzpicture}
\caption{A Cohen--Macaulay digraph}
\label{digraph1}
\end{figure}
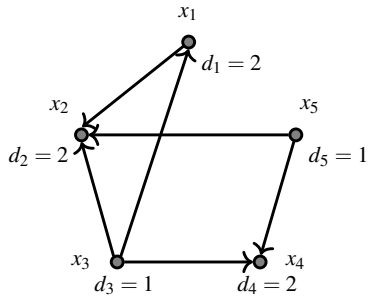
\noindent whose edge ideal is $I=I(\mathcal{D})=(x_1^2x_3,\, x_1x_2^2,\,
x_3x_2^2,\, x_3x_4^2,\, x_4^2x_5,\,  x_2^2x_5)$. By Theorem~\ref{pitones-reyes-toledo}, the irreducible
decomposition  of $I$ is 
$$
I=(x_1^2,\, x_2^2,\, x_4^2)\cap
(x_1,\, x_3,\, x_5)\cap
(x_2^2,\, x_3,\, x_4^2)\cap
(x_2^2,\, x_3,\, x_5).
$$

Using {\em Macaulay\/}$2$ \cite{mac2}, we get that $I$ is a
Cohen--Macaulay ideal whose Rees algebra is
Cohen-Macaulay and whose integral closure is 
$$
\overline{I}=I+(x_1x_2x_3,\,x_1x_3x_4,\,x_2x_3x_4,\,x_2x_4x_5).
$$

We note that the Cohen--Macaulayness of both $I$ and
its Rees algebra is destroyed (or recovered) by a single stroke of
reversing the edge orientation of $(x_5,x_2)$. This also destroys 
the unmixedness property of $I$.
\end{example}

In the summer of 2017 Antonio Campillo asked in a seminar at the University
of Valladolid if there was anything special if we take an
oriented graph $\mathcal{D}$ with underlying graph $G$ and set $d_i$
equal to $\deg_G(x_i)$ for $i=1,\ldots,n$. It will turn out that in 
determining the Cohen--Macaulay property of $\mathcal{D}$ one can
always make this canonical choice of weights. 

\begin{lemma}\label{aug12-17} Let $I\subset R$ be a monomial ideal,
let $x_i$ be a variable and let $h_1,\ldots,h_r$ be the monomials of $G(I)$ where
$x_i$ occurs. If $x_i$ occurs in $h_j$ with exponent $1$ for all $j$
and $m$ is a positive integer, then $I$ is
Cohen--Macaulay of height $g$ if and only if
$((G(I)\setminus\{h_j\}_{j=1}^r)\cup\{x_i^mh_j\}_{j=1}^r)$ is Cohen--Macaulay of
height $g$. 
\end{lemma}

\begin{proof} It follows at once from \cite[Lemmas~3.3 and
3.5]{lattice-dim1}. \qed
\end{proof}

It was pointed out to us by Ng\^{o} Vi\^{e}t Trung that the next
proposition follows from the fact that the map $x_i \to y_i^{d_i}$ (replacing $x_i$ by
$y_i^{d_i}$) defines a faithfully flat
homomorphism from $K[X]$ to $K[Y]$.

\begin{proposition} Let $I$ be a squarefree monomial ideal and let
$d_i=d(x_i)$ be a weighting of the variables. If
$G'$ is set of monomials obtained from $G(I)$ by replacing each $x_i$
with $x_i^{d_i}$, then $I$ is Cohen--Macaulay if and only if
$I'=(G')$ is Cohen--Macaulay.
\end{proposition}

\begin{proof} It follows applying Lemma~\ref{aug12-17} to each
$x_i$. \qed 
\end{proof}

If a vertex $x_i$ is a {\it sink\/} (i.e., 
has only arrows entering $x_i$), the next result shows that the
Cohen-Macaulay property of $I(\mathcal{D})$ is independent of the
weight of $x_i$.  

\begin{corollary} If $x_i$ is a sink of a weighted 
oriented graph $\mathcal{D}$ and $\mathcal{D}'$ is the digraph obtained
from $\mathcal{D}$ by replacing $d_i$ with $d_i=1$. Then
$I(\mathcal{D})$ is Cohen--Macaulay if and only if $I(\mathcal{D}')$
is Cohen--Macaulay.
\end{corollary}

That is, to determine whether or not an oriented graph $\mathcal{D}$ is
Cohen--Macaulay one may assume that all sources and sinks 
have weight $1$. In particular if all vertices of $\mathcal{D}$ are either sources of
sinks and $G$ is its underlying graph, then $I(\mathcal{D})$ is Cohen--Macaulay
if and only if $I(G)$ is Cohen--Macaulay.

Let $I$ be a monomial ideal and let $x_i$ be a fixed variable that occurs 
in $G(I)$. Let $q$ be the maximum of the degrees in $x_i$ of the
monomials of $G(I)$ and let $\mathcal{B}_i$ be the set of all monomial
of $G(I)$ of degree in $x_i$ equal to $q$. For use below we set  
$$
\mathcal{A}_i:=\{x^a\vert\, \deg_{x_i}(x^a)<q\}\cap G(I)=
G(I)\setminus\mathcal{B}_i,
$$
$p:=\max\{\deg_{x_i}(x^a)\vert\, x^a\in\mathcal{A}_i\}$ and 
$L:=(\{x^a/x_i\vert\,x^a\in\mathcal{B}_i\}\cup\mathcal{A}_i\})$.

\begin{theorem}{\rm\cite{depth-monomial}}\label{morey-vila-oaxaca-2017}
Let $I$ be a monomial ideal. If $p\geq 1$, and $q-p\geq 2$, then
$$\depth(R/I)=\depth(R/L).$$  
\end{theorem}

\begin{proof} To simplify notation we set $i=1$. We may assume that
$G(I)=\{f_1,\ldots,f_r\}$, where $f_1,\ldots,f_m$ are all the elements 
of $G(I)$ that contain $x_1^q$ and
$f_{m+1},\ldots,f_s$ are all the elements of $G(I)$ that 
contain some positive power $x_1^\ell$ of $x_1$ for some $1\leq
\ell<q$. Let $X'=\{x_{1,2},\ldots,x_{1,q-1}\}$ be a set of new
variables.
If $f=x_1^sf'$ is a monomial with $\gcd(x_1,f')=1$, we write
$f^{\rm pol}=x_{1,2}\cdots x_{1,t+1}x_1^{s-t}f'$ where
$t=\min(q-2,s)$. Making a
partial polarization of $x_1^q$ with respect to the new variables
$x_{1,2},\ldots,x_{1,q-1}$ \cite[p.~203]{monalg-rev}, gives that $f_i$ polarizes to 
$f_i^{\rm pol}=x_{1,2}\cdots x_{1,q-1}x_1^2f_i'$ for $i=1,\ldots,m$,
where $f_1',\ldots,f_m'$ are monomials that do not contain $x_1$ and
$f_i=x_1^qf_i'$ for $i=1,\ldots,m$.  
Hence, using that $q-p\geq 2$, one has the partial polarization
$$
I^{\rm pol}=(x_{1,2}\cdots x_{1,q-1}x_1^2f_1',\ldots,x_{1,2}\cdots
x_{1,q-1}x_1^2f_m',f_{m+1}^{\rm pol},\ldots,f_s^{\rm
pol},f_{s+1},\ldots,f_r),
$$
where $f_{m+1}^{\rm pol},\ldots,f_s^{\rm pol}$ do not contain $x_1$
and $I^{\rm pol}$ is an ideal of $R^{\rm
pol}=R[x_{1,2},\ldots,x_{1,q-1}]$. 
On the other hand, 
one has the partial polarization
$$
L^{\rm pol}=(x_{1,2}\cdots x_{1,q-1}x_1f_1',\ldots,x_{1,2}\cdots
x_{1,q-1}x_1f_m',f_{m+1}^{\rm pol},\ldots,f_s^{\rm
pol},f_{s+1},\ldots,f_r).
$$

By making the substitution
$x_1^{2}\rightarrow x_1$ in
each element of $G(I^{\rm pol})$ this will not affect the depth of
$R^{\rm pol}/I^{\rm pol}$ (see \cite[Lemmas~3.3 and
3.5]{lattice-dim1}). Thus 
$$
q-2+\depth(R/I)=
\depth(R^{\rm pol}/I^{\rm pol})=\depth(R^{\rm pol}/L^{\rm
pol})=q-2+\depth(R/L),
$$
and consequently $\depth(R/I)=\depth(R/L)$. \qed
\end{proof}

\begin{corollary}\label{oct20-17} Let $I=I(\mathcal{D})$ be the edge ideal of a
vertex-weighted oriented graph with vertices $x_1,\ldots,x_n$ 
and let $d_i$ be the weight of $x_i$. If $\mathcal{U}$ is the digraph
obtained from $\mathcal{D}$ by assigning weight $2$ to every vertex $x_i$
with $d_i\geq 2$, then $I$ is Cohen--Macaulay if and only 
if $I(\mathcal{U})$ is Cohen--Macaulay. 
\end{corollary}

\begin{proof} By applying Theorem~\ref{morey-vila-oaxaca-2017} to each
vertex $x_i$ of $\mathcal{D}$ of weight at least $3$, we obtain 
that $\depth(R/I(\mathcal{D}))$ is equal to
$\depth(R/I(\mathcal{U}))$. Since $I(\mathcal{D})$ and $I(\mathcal{U})$ have the
same height, then $I(\mathcal{D})$ is Cohen-Macaulay if and only if 
$I(\mathcal{U})$ is Cohen--Macaulay.  \qed
\end{proof}

\begin{lemma}{\rm \cite[Theorem~16.3(4),
p.~200]{Har}}\label{acyclic-char}  Let $\mathcal D$ be an oriented 
graph. Then $\mathcal{D}$ is {\it acyclic}, 
i.e., $\mathcal{D}$ has no oriented cycles,
if and only if there is a linear  
ordering of the vertex set $V(\mathcal{D})$ such that all the 
edges of $\mathcal D$ are of the form 
$(x_i,x_j)$ with $i<j$. 
\end{lemma}

A complete oriented graph is called a {\it tournament\/}. The next
result shows that weighted acyclic tournaments are Cohen--Macaulay.

\begin{corollary}\label{apr22-17} Let $\mathcal{D}$ be a weighted 
oriented graph. If the underlying graph $G$ of $\mathcal{D}$ is a complete graph 
and $\mathcal{D}$ has no oriented cycles, then $I(\mathcal{D})$ is
Cohen--Macaulay.
\end{corollary}

\begin{proof} By Lemma~\ref{acyclic-char}, $\mathcal{D}$ has a source
$x_i$ for some $i$. Hence $\{x_1,\ldots,x_n\}$ is not a strong vertex cover
of $\mathcal{D}$ because there is no arrow entering $x_i$. Thus, by
Corollary~\ref{prt-main}, the maximal ideal
$\mathfrak{m}=(x_1,\ldots,x_n)$ cannot be an associated prime of
$I(\mathcal{D})$. Therefore $R/I(\mathcal{D})$ has depth at least $1$.
As $\dim(R/I(\mathcal{D}))=1$, we get that $R/I(\mathcal{D})$ is
Cohen--Macaulay. \qed
\end{proof}

The next result gives an interesting family of digraphs whose edge
ideals satisfy Alexander duality. Recall that a digraph $\mathcal D$ is called {\it
transitive} if for any two edges $(x_i,x_j)$,
$(x_j,x_k)$ in $E(\mathcal{D})$ with $i,j,k$ distinct, we have 
that $(x_i,x_k)\in E(\mathcal{D})$. Acyclic tournaments are transitive
and transitive oriented graphs are acyclic.

\begin{theorem}\label{alexander-duality-transitive} 
If $\mathcal{D}$ is a transitive oriented graph
and $I=I(\mathcal{D})$ is its edge ideal, then Alexander duality holds,
that is, $I^*=I^\vee$.
\end{theorem}

\begin{proof} ``$\supset$'': Take $x^a\in G(I^\vee)$.   According to 
Theorem~\ref{pitones-reyes-toledo}, there is a strong vertex 
cover $C$ of $\mathcal{D}$ such that 
\begin{equation}\label{may27-17}
x^a=\left(\prod_{x_k\in
L_1}x_k\right)\left(\prod_{x_k\in L_2\cup L_3}x_k^{d_k}\right),
\end{equation}
where $L_i=L_i(C)$ for $i=1,2,3$. Fix a monomial $x_ix_j^{d_j}$ in
$G(I(\mathcal{D}))$, that is, $(x_i,x_j)\in E(\mathcal{D})$. It
suffices to show that $x^a$ is in the ideal 
$I_{i,j}:=(\{x_i,x_j^{d_j}\})$. If $x_i\in C$, then by Eq.~(\ref{may27-17}) the
variable $x_i$ occurs in $x^a$ because $C$ is equal to $L_1\cup
L_2\cup L_3$. Hence $x^a$ is a multiple of $x_i$ and $x^a$ is in
$I_{i,j}$, as required. Thus we may assume that $x_i\notin C$. By 
Theorem~\ref{pitones-reyes-toledo} the ideal
$$
I_C=(L_1\cup\{x_k^{d_k}\vert\, x_k\in L_2\cup L_3\})
$$
is an irreducible component of $I(\mathcal{D})$ and $x_ix_j^{d_j}\in
I_C$.

Case (I): $x_ix_j^{d_j}\in(L_1)$. Then $x_ix_j^{d_j}=x_kx^b$ for some
$x_k\in L_1$. Hence, as $x_i\notin C$, we get $j=k$. Therefore, as 
$x_j\in L_1$, there is $x_\ell\notin C$ such that $(x_j,x_\ell)$ is
in $E(\mathcal{D})$. Using that $\mathcal{D}$ is transitive gives
$(x_i,x_\ell)\in E(\mathcal{D})$
 and $x_ix_\ell^{d_\ell}\in I(\mathcal{D})$. In particular
 $x_ix_\ell^{d_\ell}\in I_C$, a contradiction because $x_i$ and
 $x_\ell$ are not in $C$. Hence this case cannot occur.

Case (II): $x_ix_j^{d_j}\in(\{x_k^{d_k}\vert\, x_k\in L_2\cup
L_3\})$. Then $x_ix_j^{d_j}=x_k^{d_k}x^b$ for some $x_k\in L_2\cup
L_3$. As $x_i\notin C$, we get $j=k$ and by Eq.~(\ref{may27-17}) we obtain $x^a\in
I_{i,j}$, as required.

``$\subset$'': Take a minimal generator $x^\alpha$ of $I^*$. By
Lemma~\ref{duality-of-exponents}, for each $i$ either $\alpha_i=1$ or
$\alpha_i=d_i$. Consider the set $A=\{x_i\vert\, \alpha_i\geq 1\}$. 
We can write $A=A_1\cup A_2$, where $A_1$ (resp. $A_2$) is the set of all $x_i$ such
that $\alpha_i=1$ (resp. $\alpha_i=d_i\geq 2$). As $(A)$ contains $I$, 
from the proof of Proposition~\ref{apr23-17}, and using
Theorem~\ref{pitones-reyes-toledo}, there exists a strong vertex cover $C$ of $\mathcal{D}$ contained
in $A$ such that the ideal 
$$ 
I_C=(L_1(C)\cup\{x_i^{d_i}\vert\, x_i\in L_2(C)\cup L_3(C)\})
$$
is an irreducible component of $I(\mathcal{D})$. Thus it suffices to
show that any monomial of $G(I_C)$ divides $x^\alpha$ because this
would give $x^a\in I^\vee$.

Claim (I): If $x_k\in A_1$, then $d_k=1$ or $x_k\in L_1(A)$. Assume
that $d_k\geq 2$. Since $x^\alpha$ is a minimal generator of $I^*$,
the monomial $x^\alpha/x_k$ is not in $I^*$. Then there is and edge
$(x_i,x_j)$ such that $x^\alpha/x_k$ is not in the ideal
$I_{i,j}:=(\{x_i,x_j^{d_j}\})$. As $x^\alpha\in I^*$ and $d_k\geq 2$, one has that $x^\alpha$
is in $I_{i,j}$ and $i=k$. Notice that $x_j$ is not in
$A_2$ because $x^\alpha/x_k$ is not in $I_{k,j}$. If 
$x_j$ is not in $A_1$ the proof is complete because $x_k\in L_1(A)$. Assume that 
$x_k$ is in $A_1$. Then $d_j\geq 2$ because $x^\alpha/x_k$ is not in
$I_{k,j}$. Setting $k_1=k$ and $k_2=j$ and applying the previous
argument to $x^\alpha/x_{k_2}$, there is $x_{k_3}\notin A_2$ 
such that $(x_{k_2},x_{k_3})$ is in $E(\mathcal{D})$. Since
$\mathcal{D}$ is transitive, $(x_{k_1},x_{k_3})$ is in
$E(\mathcal{D})$. If $x_{k_3}$ is not in $A_1$ the proof is complete.
If $x_{k_3}$ is in $A_1$, then $d_{k_3}\geq 2$ and we can continue
using the previous argument. Suppose we have constructed
$x_{k_1},\ldots,x_{k_s}$ for some $s\leq r$ such that $x_{k_s}\notin
A_2$, and $(x_{k_1},x_{k_{s-1}})$ 
and $(x_{k_{s-1}},x_{k_s})$ are in $E(\mathcal{D})$. Since
$\mathcal{D}$ is transitive, $(x_{k_1},x_{k_s})$ is in
$E(\mathcal{D})$. If $x_{k_s}$ is not in $A_1$ the proof is complete.
If $x_{k_s}$ is in $A_1$ and $s<r$, then $d_{k_s}\geq 2$ and we can continue the
process. If $x_{k_s}$ is in $A_1$ and $s=r$, that is,
$A_1=\{x_{k_1},\ldots,x_{k_r}\}$, then applying the previous argument
to $x^\alpha/x_{k_r}$ there is $x_{r+1}$ not in $A$ such that
$(x_{r},x_{r+1})$ is in $E(\mathcal{D})$. Thus by transitivity
$(x_{k_1},x_{r+1})$ is in $E(\mathcal{D})$, that is, $x_{k_1}$ is in
$L_1(A)$.

Claim (II): If $x_k\in A_2$, then $x_k\in L_2(A)$. Since $x^\alpha\in
G(I^*)$ and $\alpha_k=d_k\geq 2$, there is $(x_i,x_k)$ in 
$E(\mathcal{D})$ such that $x^\alpha/x_k$ is not in
$I_{i,k}=(\{x_i,x_k^{d_k}\})$. In particular $x_i$ is not in $A$. To 
prove that $x_k$ is in $L_2(A)$ 
it suffices to show that $x_k$ is not in $L_1(A)$. If $x_k$ is in
$L_1(A)$, there is $x_j$ not in $A$ such that $(x_k,x_j)$ is in
$E(\mathcal{D})$. As $\mathcal{D}$ is transitive, we get that
$(x_i,x_j)$ is in $E(\mathcal{D})$ and $A\cap\{x_i,x_j\}=\emptyset$, 
a contradiction because $(A)$ contains $I$.

Take a monomial $x_k^{a_k}$ of $G(I_C)$. 

Case (A): $x_k\in L_1(C)$. Then $a_k=1$. There is $(x_k,x_j)\in E(\mathcal{D})$ with
$x_j\notin C$. Notice $x_k\in A_1$. Indeed if $x_k\in A_2$, then $x_k$
is in $L_2(A)$ because of Claim (II). Then there is $(x_i,x_k)$ in
$E(\mathcal{D})$ with $x_i\notin A$. By transitivity $(x_i,x_j)\in
E(\mathcal{D})$ and $\{x_i,x_j\}\cap C=\emptyset$, a contradiction
because $(C)$ contains $I$. Thus $x_k\in A_1$, that is, $\alpha_k=1$.
This proves that $x_k^{a_k}$ divides $x^\alpha$.

Case (B): $x_k\in L_2(C)$. Then $x_k^{a_k}=x_k^{d_k}$. First assume 
$x_k\in A_1$. Then, by Claim (I), $d_k=1$ or $x_k\in L_1(A)$. Clearly
$x_k\notin L_1(A)$ because $L_1(A)\subset L_1(C)$ and $x_k$---being
in $L_2(C)$---cannot be in $L_1(C)$. Thus $d_k=1$ and $x_k^{d_k}$
divides $x^\alpha$. Next assume $x_k\in A_2$. Then, by construction of $A_2$, 
$x_k^{d_k}$ divides $x^\alpha$.  

Case (C): $x_k\in L_3(C)$. Then $x_k^{a_k}=x_k^{d_k}$. First assume 
$x_k\in A_1$. Then, by Claim (I), $d_k=1$ or $x_k\in L_1(A)$. Clearly
$x_k\notin L_1(A)$ because $L_1(A)\subset L_1(C)$ and $x_k$---being
in $L_3(C)$---cannot be in $L_1(C)$. Thus $d_k=1$ and $x_k^{d_k}$
divides $x^\alpha$. Next assume $x_k\in A_2$. Then, by construction of $A_2$, 
$x_k^{d_k}$ divides $x^\alpha$.  \qed
\end{proof}

\begin{corollary}\label{jun1-17} If $\mathcal{D}$ is a weighted
acyclic tournament, then $I(\mathcal{D})^*=I(\mathcal{D})^\vee$, that is, Alexander duality
holds. 
\end{corollary}

\begin{proof} The result follows readily from
Theorem~\ref{alexander-duality-transitive} because acyclic
tournaments are transitive.  \qed
\end{proof}

\begin{example} Let $\mathcal{D}$ be the weighted oriented graph whose
edges and weights are 
$$(x_2,x_1),(x_3,x_2),(x_3,x_4),(x_3,x_1),$$
and $d_1=1,d_2=2,d_3=1, d_4=1$, respectively. This digraph is transitive. Thus 
$I(\mathcal{D})^*=I(\mathcal{D})^\vee$. 
\end{example}

\begin{example}\label{example-alex-duality} The irreducible decomposition of the ideal 
$I=(x_1x_2^2,x_1x_3^2,x_2x_3^2)$ is 
$$
I=(x_1,x_2)\cap(x_1,x_3^2)\cap(x_2^2,x_3^2),
$$
in this case
$I^\vee=(x_1x_2,x_1x_3^2,x_2^2x_3^2)=(x_1,x_2^2)\cap(x_1,x_3^2)\cap(x_2,x_3^2)=I^*$.
\end{example}

\begin{example}\label{example-alex-duality-1} The irreducible decomposition of the ideal 
$I=(x_1x_2^2,x_3x_1^2,x_2x_3^2)$ is 
$$
I=(x_1^2,x_2)\cap(x_1,x_3^2)\cap(x_2^2,x_3)\cap(x_1^2,x_2^2,x_3^2),
$$
in this case
$I^\vee=(x_1^2x_2,x_1x_3^2,x_2^2x_3)\subsetneq
(x_1,x_2^2)\cap(x_3,x_1^2)\cap(x_2,x_3^2)=I^*$.
\end{example}

\begin{example} The irreducible decomposition of the ideal 
$I=(x_1x_2^2,x_1^2x_3)$ is 
$$
I=(x_1)\cap(x_1^2,x_2^2)\cap(x_3,x_2^2),
$$
in this case $I^\vee=(x_1,x_2^2x_3)\supsetneq 
I^*=(x_1,x_2^2)\cap(x_1^2,x_3)=(x_1^2,x_1x_3,x_2^2x_3)$.
\end{example}

We come to the main result of this section.

\begin{theorem}\label{cm-weighted-oriented-trees}
Let $\mathcal{D}$ be a weighted oriented forest without isolated
vertices and let $G$ be its underlying forest. The following
conditions are equivalent:
\begin{enumerate}
\item[$(a)$] $\mathcal{D}$ is Cohen--Macaulay.
\item[$(b)$] $I(\mathcal{D})$ is unmixed, that is, all its associated
primes have the same height.
\item[$(c)$] $G$ has a 
perfect matching $\{x_1,y_1\},\ldots,\{x_r,y_r\}$ so that 
$\deg_G(y_i)=1$ for $i=1,\ldots, r$ and $d(x_i)=d_i=1$ if
$(x_i,y_i)\in E(\mathcal{D})$.
\end{enumerate}
\end{theorem}

\begin{proof} It suffices to show the result when $G$ is connected, 
that is, when $\mathcal{D}$ is an oriented tree. Indeed $\mathcal{D}$
is Cohen--Macaulay (resp. unmixed) if and only if all connected
components of $\mathcal{D}$ are Cohen--Macaulay (resp. unmixed)
\cite{PRT,Vi2}. 
 
$(a)$ $\Rightarrow$ $(b)$: This implication follows from the general fact that
Cohen--Macaulay graded ideals are unmixed \cite[Corollary~3.1.17]{monalg-rev}.  

$(b)$ $\Rightarrow$ $(c)$: According to the results of \cite{Vi2} one has
that  $|{V(G)}|=2r$ and $G$ has a 
perfect matching $\{x_1,y_1\},\ldots,\{x_r,y_r\}$ so that 
$\deg_G(y_i)=1$ for $i=1,\ldots,r$. Consider the oriented graph
$\mathcal{H}$ with vertex set $V(\mathcal{H})=\{x_1,\ldots,x_r\}$
whose edges are all $(x_i,x_j)$ such that $(x_i,x_j)\in
E(\mathcal{D})$. As $\mathcal{H}$ is acyclic, by
Lemma~\ref{acyclic-char}, we may assume that the vertices of
$\mathcal{H}$ have a ``topological'' order, that is, if $(x_i,x_j)\in
E(\mathcal{H})$, then $i<j$. If 
$(y_i,x_i)\in E(\mathcal{D})$ for $i=1,\ldots,r$, there is nothing to
prove. Assume that $(x_k,y_k)\in
E(\mathcal{D})$ for some $k$. To complete the proof we need only 
show that $d(x_k)=d_k=1$. We proceed by
contradiction assuming that $d_k\geq 2$. In particular $x_k$ cannot be
a source of $\mathcal{H}$. Setting
$X=\{x_1,\ldots,x_r\}$, consider the set of vertices
$$
C=(X\setminus N_\mathcal{H}^-(x_k))\cup\{y_i\vert\, x_i\in
N_\mathcal{H}^-(x_k)\}\cup\{y_k\},
$$
where $N_\mathcal{H}^-(x_k)$ is the {\it in-neighbor} set of $x_k$ 
consisting of all $y\in V(\mathcal{H})$ such that $(y,x_k)\in
E(\mathcal{H})$. Clearly $C$ is a vertex cover of $G$ with $r+1$
elements because the set $N_\mathcal{H}^-(x_k)$ is an independent set
of $G$. Let us show that $C$ is a strong cover of $\mathcal{D}$. 
The set $N_\mathcal{H}^-(x_k)$ is not empty because $x_k$ is not a
source of $\mathcal{D}$. Thus $x_k$ is not in $L_3(C)$. Since 
$L_3(C)\subset\{x_k,y_k\}\subset C$, we get $L_3(C)=\{y_k\}$. There is
no arrow of $\mathcal{D}$ with source at $x_k$ and head outside of
$C$, that is, $x_k$ is in $L_2(C)$. Hence $(x_k,y_k)$ is in
$E(\mathcal{D})$ with $x_k\in L_2(C)$ and $d(x_k)\geq 2$. This means
that $C$ is a strong cover of $\mathcal{D}$. Applying
Theorem~\ref{prt-main} gives that $\mathfrak{p}=(C)$ is an associated
prime of $I(\mathcal{D})$ with $r+1$ elements, a contradiction because
$I(\mathcal{D})$ is an unmixed ideal of height $r$.

$(c)$ $\Rightarrow$ $(a)$: We proceed by induction on $r$. The case $r=1$ is
clear because $I(\mathcal{D})$ is a principal ideal, hence
Cohen--Macaulay. Let $\mathcal{H}$ be the graph defined in the proof
of the previous implication. As
before we may assume that 
the vertices of $\mathcal{H}$ are in topological order and we set
$R=K[x_1,\ldots,x_r,y_1,\ldots,y_r]$. 

Case (I): Assume that $(y_r,x_r)\in E(\mathcal{D})$. Then $x_r$ is a sink 
of $\mathcal{D}$ (i.e., has only arrows entering $x_r$). Using the
equalities
$$
(I(\mathcal{D})\colon x_r^{d_r})=(N_G(x_r),I(\mathcal{D}\setminus
N_G(x_r)))\ \mbox{ and }\
(I(\mathcal{D}),x_r^{d_r})=(x_r^{d_r},I(\mathcal{D}\setminus\{x_r\})),
$$
and applying the induction hypothesis to $I(\mathcal{D}\setminus
N_G(x_r))$ and $I(\mathcal{D}\setminus\{x_r\})$ we obtain that the
ideals $(I(\mathcal{D})\colon x_r^{d_r})$ and
$(I(\mathcal{D}),x_r^{d_r})$ are Cohen--Macaulay of dimension $r$.
Therefore, as $I(\mathcal{D})$ has height $r$, from the
exact sequence  
$$
0\rightarrow R/(I(\mathcal{D})\colon
x_r^{d_r})[-d_r]\stackrel{x_r^{d_r}}{\rightarrow}
R/I(\mathcal{D})\rightarrow R/(I(\mathcal{D}),x_r^{d_r})\rightarrow 0
$$
and using the depth lemma (see \cite[Lemma~2.3.9]{monalg-rev}) we obtain that $I(\mathcal{D})$ is
Cohen--Macaulay.

Case (II): Assume that $(x_r,y_r)\in E(\mathcal{D})$. Then
$d(x_r)=d_r=1$ and $x_ry_r^{e_r}\in I(\mathcal{D})$, 
where $d(y_r)=e_r$. Using the
equalities
$$
(I(\mathcal{D})\colon x_r)=(N_G(x_r)\setminus\{y_r\},y_r^{e_r},I(\mathcal{D}\setminus
N_G(x_r)))\ \mbox{ and }\
(I(\mathcal{D}),x_r)=(x_r,I(\mathcal{D}\setminus\{x_r\})),
$$
and applying the induction hypothesis to $I(\mathcal{D}\setminus
N_G(x_r))$ and $I(\mathcal{D}\setminus\{x_r\})$ we obtain that the
ideals $(I(\mathcal{D})\colon x_r)$ and
$(I(\mathcal{D}),x_r)$ are Cohen--Macaulay of dimension $r$.
Therefore, as $I(\mathcal{D})$ has height $r$, from the
exact sequence  
$$
0\rightarrow R/(I(\mathcal{D})\colon
x_r)[-1]\stackrel{x_r}{\rightarrow}
R/I(\mathcal{D})\rightarrow R/(I(\mathcal{D}),x_r)\rightarrow 0
$$
and using the depth lemma \cite[Lemma~2.3.9]{monalg-rev} we obtain that $I(\mathcal{D})$ is
Cohen--Macaulay. \qed
\end{proof}

The following result was conjectured in a preliminary version of this
paper and proved recently in \cite{Ha-Lin-Morey-Reyes-Vila} using 
polarization of monomial ideals.

\begin{theorem}{\rm\cite[Theorem~3.1]{Ha-Lin-Morey-Reyes-Vila}}\label{cm-weighted-oriented-whiskers}
Let $\mathcal{D}$ be a weighted oriented graph and let $G$ be its
underlying graph. Suppose that $G$ has a 
perfect matching $\{x_1,y_1\},\ldots,\{x_r,y_r\}$ where
$\deg_G(y_i)=1$ for each $i$. The following 
conditions are equivalent:
\begin{enumerate}
\item[$(a)$] $\mathcal{D}$ is Cohen--Macaulay.
\item[$(b)$] $I(\mathcal{D})$ is unmixed, that is, all its associated
primes have the same height.
\item[$(c)$] $d(x_i)=1$ for any edge of $\mathcal{D}$ of the form 
$(x_i,y_i)$.
\end{enumerate}
\end{theorem}

The equivalence between $(b)$ and $(c)$ was also proved in
\cite[Theorem~4.16]{PRT}.

\begin{remark} If $\mathcal{D}$ is a Cohen--Macaulay weighted 
oriented graph, then $I(\mathcal{D})$ is unmixed and 
${\rm rad}(I(\mathcal{D}))$ is Cohen--Macaulay. This follows from 
the fact that Cohen--Macaulay ideals are unmixed and using a result of Herzog, Takayama and Terai
\cite[Theorem~2.6]{herzog-takayama-terai} which is valid for any
monomial ideal. It is an open question whether the converse is true
\cite[Conjecture~5.5]{PRT}.
\end{remark}

\begin{example} The radical of the ideal $I=(x_2x_1,x_3x_2^2,x_3x_4)$
is Cohen--Macaulay and $I$ is not unmixed. The irreducible components
of $I$ are $(x_1,x_3)$, $(x_2,x_3)$, $(x_1,x_2^2,x_4)$, $(x_2,x_4)$.
\end{example}

\begin{example}{(Terai)} The ideal $I=(x_1,x_2)^2\cap(x_2,x_3)^2\cap(x_3,x_4)^2$
is unmixed, ${\rm rad}(I)$ is Cohen-Macaulay, and $I$ is not 
Cohen--Macaulay.
\end{example}

\begin{acknowledgement} We would like to thank Ng\^{o} Vi\^{e}t Trung
and the referees for a careful reading of the paper and for the improvements suggested. 
The first, third and fourth authors were partially supported by
the Spanish {\it Ministerio de Econom\'\i a y Competitividad} grant
MTM2016-78881-P. 
The second and fourth authors
were supported by SNI. 
The fifth author was
supported by a scholarship from CONACYT
\end{acknowledgement}

\bibliographystyle{plain}

\end{document}